\documentclass[final]{elsarticle}
\usepackage{indentfirst}
\usepackage{hyperref}
\usepackage{amsmath}
\usepackage[textwidth=16cm,textheight=24cm,left=2.5cm,right=2.5cm,top=3cm,bottom=3cm]{geometry}
\usepackage{psfrag}
\usepackage{amsmath}
\usepackage{amsfonts}
\usepackage{chemarrow}
\usepackage{algorithm}
\usepackage{algorithmic}
\usepackage{mathrsfs}
\usepackage{amsthm}
\usepackage{tikz}
\usetikzlibrary{decorations.markings}
\usetikzlibrary{calc}
\usetikzlibrary{arrows,calc,shapes,decorations.pathreplacing}
\usepackage[active]{srcltx}
\usepackage{amssymb,enumerate}
\usepackage{color}

\newtheorem{theorem}{Theorem}

\newtheorem{lem}{Lemma}

\theoremstyle{definition}

\newcommand{\Romannum}[1]{\uppercase\expandafter{\romannumeral #1}}

\numberwithin{equation}{section}
\newcommand{\LL}{\mathcal{L}}

\newcommand{\PP}{\mathbb{P}}
\begin{document}

\begin{frontmatter}


\title{Homomorphism of independent random variable convolution and matrix multiplication}
\author[fzu]{Yue Liu}\ead{liuyue@fzu.edu.cn}

\address[fzu]{School of Mathematics and Statistics, Fuzhou University, Fuzhou, 350116, China}

\date{}

\begin{abstract}
  A map is given showing that convolutions of independent random variables over a finite group and matrix multiplications of doubly stochastic matrices are homomorphic. As an application, a short proof is given to the theorem that the limiting distributions of stochastic processes with stationary independent increments over a finite group are always uniform.

\end{abstract}
  

\end{frontmatter}

\section*{The Result}
Convolution of random variables is a basic operation in probability theory, especially the convolutions of independent ones. In this paper, attention is restricted to convolutions of independent random variables over finite groups.

Let $G = \{g_1,\ldots,g_n\}$ be a finite group. By Caley Theorem, let \[\phi:\ G  \rightarrow  \Sigma,\ g_k\mapsto \sigma_k\] be an isomorphism, where $\sigma_k$ is the left transition of $G$ with the element $g_k$, thus
$\Sigma = \{\sigma_1,\ldots,\sigma_n\}$ can be also regarded as a permutation subgroup of the symmetric group $S_n$ over $[n]$. For each $k$, there is an order $n$ permutation matrix $P_{\sigma_k}=\left(p_{ij}^{(k)}\right)_{n\times n} $ corresponding to $\sigma_k$, where
\[p_{ij}^{(k)} = \left\{\begin{array}{ll}
  1, & i = \sigma_k(j),\\
  0, & \text{otherwise.}
\end{array} \right. \]
 
Let $X$ be a random variable over $G$. Write $p_k = \PP[X = g_k],\ k =1,\ldots,n$. Define the {\bf \emph{convolution matrix}} of $X$, denoted by $Con(X)$, as
\[Con(X) = \sum_{k=1}^n p_k P_{\sigma_k}.\]
It is easy to see that $Con(X)$ is a doubly stochastic matrix. The following lemma shows that convolutions of independent random variables over a finite group and matrix multiplications of doubly stochastic matrices are homomorphic.

\begin{lem}\label{lem:conv_matrix}
  Let $X,Y$ be two independent random variables over a finite group $G$. Then 
  \[
   Con(X\cdot Y)=Con(X)Con(Y),
   \] 
   where $ X\cdot Y$ is the convolution of $X$ and $Y$.
\end{lem}
\begin{proof}
  The lemma can be verified by a direct computation.
\end{proof}

Random walk is a typical stochastic process. Random walks with stationary and independent increments over Euclidean spaces are called \emph{L\'{e}vy processes} (see \cite{book_levy_process} for example). A more general framework for the study of such processes is proposed in \cite{book_random_walk_semigroups},  where the random variables are assumed to be over any topological semigroups, rather than just Euclidean spaces. It was shown that the asymptotic behaviors of the processes over compact spaces and noncompact spaces are quite different (see Chapter 2 and 3  in \cite{book_random_walk_semigroups}).

As an application of Lemma \ref{lem:conv_matrix}, together with the well-knowm Perron Theorem (Theorem \ref{thm-Perron}), we give a short proof to the following Theorem \ref{thm-limiting}, which characterize  the asymptotic behavior of  random walks with stationary and independent increments over finite groups.

\begin{theorem} (Perron Theorem, \cite[Theorem 8.4.4]{book_Mat_Analy_Horn_Johnson_2012}) \label{thm-Perron}
	Let $A \in M_n$ be irreducible and nonnegative, and suppose that $n
	\ge 2$. Then
	\begin{enumerate}
		\item $\rho(A)>0$.
		\item $\rho(A)$ is an algebraically simple eigenvalue of $A$.
		\item there is a unique real vector $x=[x_i]$ such that $Ax=\rho(A)x$ and $x_1+\cdots+x_n =1$; this vector is positive.
		\item there is a unique real vector $y=[y_i]$ such that $y^TA=\rho(A)y^T$ and $x_1y_1+\cdots+x_n y_n = 1$; this vector is positive.
	\end{enumerate}
\end{theorem}

Let $X$ be a  random variable. $\LL(X)$ is used to denote the distribution law of $X$. 

\begin{theorem}\label{thm-limiting}
    Let $G=\{g_1,\ldots,g_n\}$ be a finite group. $(X_m)_{m\ge 1}$ be a $G$-valued stochastic process such that $X_1 = \xi_1$ and $
    X_{m+1} = X_m\xi_{m+1},\ m\ge 1,
    $ where $(\xi_m)_{m\ge 1}$ are i.i.d. random variables whose supports are $G$. Then \[\lim_{m\to \infty}\LL(X_m)= \LL\] exists, and $\LL$ is convolution invariant, i.e., $\LL$ is a uniform distribution over $G$. 
  \end{theorem}
  \begin{proof}
    Let $P_{\sigma_1},\ldots, P_{\sigma_n}$ be the permutation matrices as in the definition of convolution matrices. Write $P_{\sigma_k} = (p_{ij}^{(k)})$. By the definitions, $p_{ij}^{(k)} = 1$ means $g_k = g_j\cdot g_i^{-1}$. Thus, for every pair of $i,j\in [n]$, there exists only one $k\in[n]$ such that $p_{ij}^{(k)} = 1$. Then $P_{\sigma_1},\ldots, P_{\sigma_n}$ are linearly independent, and $\sum_{k=1}^n P_{\sigma_k}$ is the matrix $J$ whose entries are all 1.
  
    Since $(\xi_m)_{m\ge 1}$ are i.i.d. random variables, their convolution matrices are the same. Write $A = Con(\xi_m)$. Since the support of $\xi_m$ is $G$, every $p_k = \PP[\xi_m=g_k]$ is positive for $k=1,\ldots,n$, which means $A$ is a positive matrix.
  
    By Lemma \ref{lem:conv_matrix}, $Con(X_m)=A^m$. Since $P_{\sigma_1},\ldots, P_{\sigma_n}$ are linearly independent, $\lim_{m\to \infty}\LL(X_m)$ exists is equivalent to $\lim_{m\to \infty} A^m$ exists, and the limiting distribution is uniquely determined by $\lim_{m\to \infty} A^m$.
  
    Since $A$ is a positive doubly stochastic matrix, by Perron Theorem, the spectral radius $\rho(A)=1$ is an algebraically simple eigenvalue, and $x =\frac{1}{n} [1,\ldots,1]^T, y^T = [1,\ldots,1]$ are the unique left and right eigenvectors corresponding to $\rho(A)$ as described in Theorem \ref{thm-Perron}, respectively. Then
    \[
    \lim_{m\to \infty} A^ m =x\cdot y^T= \frac{1}{n}J = \frac{1}{n}P_{\sigma_1}+\cdots+\frac{1}{n}P_{\sigma_n},
    \] yielding that the limiting distribution of $(X_m)_{m\ge 1}$ is the uniform distribution.
  \end{proof} 
  \section*{Acknowledgments}
  The author would like to thank Prof. Jian Wang of Funjian Normal University for the helpful discussions and suggestions, especially for providing the background of random walks.

\bibliographystyle{elsarticle-num-names}

\bibliography{convolution}

\end{document}